\documentclass[10pt]{amsart}
\usepackage{latexsym, amsmath, amssymb}
\usepackage{version}
\usepackage{epsfig,graphics}
\usepackage{color}
\usepackage[T1]{fontenc}
\usepackage{amsthm, amsopn, amsfonts}

\newcommand{\cal}{\mathcal}

\newcommand{\newsection}[1]
{\section{#1}\setcounter{theorem}{0} \setcounter{equation}{0}
\par\noindent}

\newtheorem{theorem}{Theorem}

\newtheorem{lemma}[theorem]{Lemma}
\newtheorem{corr}[theorem]{Corollary}

\newcommand{\R}{{\mathbb R}}
\newcommand{\Z}{{\mathbb Z}}

\newcommand{\ang}{{\not\negmedspace\nabla}}

\newcommand{\e}{\epsilon}
\newcommand{\rs}{{r^*}}

\renewcommand{\S}{{\mathbb S}}
\newcommand{\M}{\cal M}
\newcommand{\tv}{{\tilde{v}}}
\newcommand{\tphi}{{\tilde{\phi}}}

\newcommand{\ts}{{\tilde{s}}}

\newcommand{\weight}{{\Bigl(1-\frac{2M}{r}\Bigr)}}

\begin{document}

\title
{ Local energy estimate on Kerr black hole backgrounds }

\author{Daniel Tataru}
\author{Mihai Tohaneanu}

\address{Department of Mathematics, University of California,
  Berkeley, CA 94720-3840}

\address{Department of Mathematics, University of California,
  Berkeley, CA 94720-3840}

\thanks{Both authors were
  supported in part by NSF grant DMS0801261.}
\begin{abstract}
  We study dispersive properties for the wave equation in the Kerr
  space-time with small angular momentum. The main result of this
  paper is to establish uniform energy bounds and local energy decay
  for such backgrounds. This follows a similar result for the
  Schwarzschild space-time obtained in earlier work \cite{MMTT}
by the authors and collaborators.

\end{abstract}

\maketitle

\newsection{Introduction}

 The aim of this article is to study the decay properties of solutions
 to the wave equation in the Kerr space-time, which describes a
 rotating black hole. Until recently even the problem of obtaining 
uniform bounds for such solutions was completely open, and 
only some partial results were obtained in \cite{FKSY}. We also refer the reader to 
related independent work in \cite{DR3}.  However, the techniques used in 
these papers are of a different flavor, as they do not carry out such
a precise analysis of the dynamics near the trapped set.

Our aim here is to establish global in time energy bounds for the wave
equation in the Kerr space-time, as well as a local energy decay
estimate. These bounds apply in the full region outside the event
horizon, as well as in a small neighborhood on the inside of the
event horizon. 

The starting point in our analysis is the earlier work \cite{MMTT} of
the authors and collaborators, which establishes similar bounds for
the wave equation in the Schwarzschild space-time. The idea is to treat
the Kerr geometry as a small perturbation of the Schwarzschild
geometry, and then adapt the methods in \cite{MMTT}. Consequently in
this article we are only considering Kerr black hole backgrounds with
small angular momentum, which are close to the Schwarzschild space-time.
Nevertheless, we are confident that our methods will carry over also
to the case of large angular momentum.

Another goal of the earlier article  \cite{MMTT} was to establish
Strichartz estimates in the Schwarzschild space-time. We also aim to 
consider the similar problem for the Kerr space-time. However,
this requires very different technical tools, and will be considered
in a subsequent paper.

The local energy estimate in \cite{MMTT} is proved using the
multiplier method; the delicate issue there is to show that a suitable
multiplier can be found.  This method is quite robust under small
perturbations of the metric, and for the most part it easily carries
over to the Kerr backgrounds with small angular momentum.  There is
however one region where this does not apply, precisely near the
photon sphere $r=3M$ (which contains all\footnote{except of course for
  the rays along the event horizon, which are not relevant to this
  discussion} the trapped periodic geodesics in the Schwarzschild
space-time).  Hence most of the new analysis here is devoted 
to understanding what happens there.

The paper is organized as follows. In the next section we discuss the
classical local energy decay decay estimates in the Minkowski
space-time and small perturbation thereof.  Then we provide a brief
overview of the local energy estimates proved in \cite{MMTT} for the
Schwarzschild space-time, along with a discussion of the relevant
geometrical issues.  Finally, the last section contains the
description of the Kerr space-time and all the new results.
Our main local energy estimate is contained in Theorem~\ref{Kerr}.
This is complemented by higher order bounds in Theorem~\ref{highreg}.

\section{Local energy decay in the Minkowski space-time}

 In the Minkowski space-time $\R^{3+1}$, consider the wave equation with
 constant coefficients
 \begin{equation}
 \Box u = f , \qquad u(0)=u_0, \qquad \partial_t u(0)=u_1
 \label{boxinhom}\end{equation}
Here $\Box  = \partial_{t} ^2 - \Delta$.
More generally, let
\[
\Box_g = \frac{1}{\sqrt{-g}}\partial_{i} (\sqrt{-g}g^{ij}\partial_{j})
\]
be the usual d'Alembertian associated to a Lorentzian metric $g$.

The seminal estimate of Morawetz~\cite{M} asserts that for solutions
to the homogeneous equation $\Box u = 0$ we have the estimate
 \begin{equation}\label{Morawetz}
  \int_{\R} \int_{\R^3} \frac{1}{|x|}|\ang u|^2(t,x)\:dx\:dt 
+ \int_{\R} |u(t,0)|^2 dt \lesssim
 \|\nabla u_0\|_{L^2}^2 + \| u_1\|_{L^2}^2
\end{equation}
where $\ang$ denotes the angular derivative. This is obtained
combining energy estimates with the multiplier method.  The radial
multiplier $Q u = (\partial_r + \frac{1}r) u$ is used, where $r$
denotes the radial variable.

Within dyadic spatial regions one can control the full gradient
$\nabla u$, but the square summability wth respect to dyadic scales is
lost. Precisely,  we define the $L^2$ local energy norm
\begin{equation}
 \|u\|_{LE_M}= \sup_{j \in \Z}
2^{-\frac{j}2}\|u\|_{L^2 (\R \times \{|x|\in [2^{j-1},
2^j]\})}
\label{lem}\end{equation}
and its $H^1$ counterpart
\begin{equation}
\|u\|_{LE^1_M}=  \|\nabla_{x,t} u\|_{LE_M}+ 
\| |x|^{-1} u\|_{LE_M}
\label{lem1}\end{equation}
For the inhomogeneous term we use the dual norm
\begin{equation}
 \|f\|_{LE_M^*}= \sum_{k \in \Z}
2^{\frac{k}2}\|f\|_{L^2 (\R \times \{|x|\in [2^{k-1}, 2^k]\})}
\label{le*m}\end{equation}

Then we have the following scale invariant local energy estimate for
solutions $u$ to the inhomogeneous equation \eqref{boxinhom}:
\begin{equation}\label{localenergyflat}
\|\nabla u\|_{L^{\infty}_t L^2_x} + \| u\|_{LE^1_M}
 \lesssim \|\nabla u_0\|_{L^2} + 
\| u_1\|_{L^2} + \|f\|_{LE_M^*+L^1_t L^2_x}
\end{equation}
This is proved using a small variation of Morawetz's method, with
multipliers of the form $a(r) \partial_r + b(r)$ where $a$ is
positive, bounded and increasing.

There are many similar results obtained in the case of
perturbations of the Minkowski space-time; see, for example, \cite{M},
\cite{KSS}, \cite{KPV}, \cite{SmSo},\cite{St}, \cite{Strauss},
\cite{Al1}, \cite{MS}. Relevant to us is the case of small long range
perturbations of the Minkowski space-time, considered in \cite{MT}.
The metrics $g$ in $\R^{3+1}$ considered there satisfy
\begin{equation}
\sum_{k \in \Z} \sup_{|x|\in [2^{k-1}, 2^k]} 
| g(t,x) - g_{M}| + |x||\nabla_{x,t} g(t,x)|+|x|^2 |\nabla^2_{x,t} g(t,x)|
\leq \epsilon
\label{gm}\end{equation}
where $g_{M}$ stands for the Minkowski metric. Then as 
a special case of the results in \cite{MT} we have

\begin{theorem}{\cite{MT}}
Let $g$ be a Lorenzian metric in $\R^{3+1}$ which satisfies
\eqref{gm} with $\epsilon$ small enough. Then the solution
$u$ to the inhomogeneous problem
 \begin{equation}
 \Box u = f , \qquad u(0)=u_0, \qquad \partial_t u(0)=u_1
 \label{boxginhom}\end{equation}
satisfies the estimate \eqref{gm}.
\end{theorem}

No general such results are known for large perturbations, where on
one hand trapping for large frequencies and on the other hand
eigenvalues and resonances for small frequencies create major
difficulties. The Schwarzschild and Kerr metrics are such large
perturbations where trapping plays a major role.

\newsection{Local energy decay in the Schwarzschild space-time}

 In the original coordinates the Schwarzschild space-time is given as a metric whose line
element is (for $I = \R\times (2M,\infty)\times \S^2$)
\begin{equation}
ds^2=-\weight dt^2+\weight^{-1}dr^2+r^2d\omega^2
\label{swm}
\end{equation}
where $d\omega^2$ is the measure on the sphere $\S^2$ and $t,r$ are
the time, respectively the radius of the $\S^2$ spheres. This metric 
is well defined in two regions,
\[
I = \R\times (2M,\infty)\times \S^2, \qquad II = \R\times (0,2M)\times \S^2
\]
  Let $\Box_S$ denote the associated d'Alembertian.

 The singularity at $r=0$ is a true metric singularity. However, the 
  singularity at the event horizon $r=2M$ is an apparent singularity 
that can be removed by a different choice of coordinates. Following
\cite{HE}, let
\[
 r^*=r+2M\log(r-2M)-3M-2M\log M
\]
 and let $v=t+r^*$. In the new coordinates $(r,v,\omega)$ the metric
 becomes
\[
 ds^2=-\weight dv^2+2 dvdr+r^2d\omega^2
\]
and can be extended to a larger manifold $I\cup II$. Moreover, if
$w=t-r^*$, one can introduce global nonsingular coordinates (all the
way to $r=0$) by rewriting the metric in the Kruskal-Szekeres
coordinate system,
\[
v' = e^{\frac{v}{4M}}, \qquad w' = -e^{-\frac{w}{4M}}.
\]

There are two places where trapping occurs on the Schwarzschild
manifold.  The first is at the event horizon $r=2M$, where the trapped
geodesics are the vertical ones in the $(r,v,\omega)$ coordinates.
However, this family of trapped rays turns out to cause no difficulty
in the decay estimates since the energy decays exponentially along it
as $v \to \infty$. The second family of trapped rays occurs on the
surface $r=3M$ which is called the photon sphere.  Null geodesics
which are initially tangent to the photon sphere will remain on the
surface for all times. Unlike the previous case, the energy is
conserved for waves localized along such rays.  However, what makes
local energy decay estimates at all possible is the fact that the
trapped rays on the photon sphere are hyperbolic.

 The $(r,v,\omega)$ coordinates are nonsingular on the event
horizon, but have the disadvantage that the level sets of $v$ are
null surfaces. This is why it is more convenient to introduce
\[
 \tv = v - \mu(r)
\]
where $\mu$ is a smooth function of $r$. In the $(\tv,r,\omega)$
coordinates the metric has the form
\[
\begin{split}
ds^2=&\ -\weight d\tv^2 +2\left(1-\weight\mu'(r)\right) d\tv dr \\ &\ +
    \Bigl(2 \mu'(r) - \weight (\mu'(r))^2\Bigr)  dr^2  +r^2d\omega^2.
\end{split}
\]
On the function $\mu$ we impose the following two conditions:

(i) $\mu (r) \geq  \rs$ for $r > 2M$, with equality for $r >
{5M}/2$.

(ii)  The surfaces $\tv = const$ are space-like, i.e.
\[
\mu'(r) > 0, \qquad 2 - \weight \mu'(r) > 0.
\]
\noindent The first condition (i) insures that the $(r,\tv,\omega)$
coordinates coincide with the $(r,t,\omega)$ coordinates in $r
>{5M}/2$. This is convenient but not required for any of our
results. What is important is that in these coordinates the metric
is asymptotically flat as $r \to \infty$ according to \eqref{gm}.

Given $0 < r_e <2M$ we consider the wave equation 
\begin{equation}
 \Box_S u = f 
 \label{boxsinhom}\end{equation}
in the cylindrical region
\begin{equation}
 \M_{R} =  \{ \tv \geq 0, \ r \geq r_e \} 
\label{mr}\end{equation}
with initial data on the space-like surface
\begin{equation}
 \Sigma_R^- =  \M_{R} \cap \{ \tv = 0 \}
\label{mr-}\end{equation}
The lateral boundary of $\M_R$, 
\begin{equation}
 \Sigma_R^+ =   \M_R \cap \{ r = r_e\} 
\label{mr+}\end{equation}
is also space-like, and can be thought of as the exit surface
for all waves which cross the event horizon. 

We define the initial (incoming) energy on $\Sigma_R^-$ as 
\begin{equation}\label{energy1}
 E[u](\Sigma_R^-) =  \int_{\Sigma_R^-} \left( |\partial_r u|^2 +
|\partial_\tv u|^2    + |\ang u|^2 \right) r^2  dr  d\omega  
\end{equation}
the outgoing energy on $\Sigma_R^+$ as 
\begin{equation}\label{energy2}
 E[u](\Sigma_R^+) = \int_{\Sigma_R^+}
 \left(  |\partial_r u|^2 +   |\partial_\tv u|^2    +
|\ang u|^2 \right) r_e^2   d\tv d\omega
\end{equation}
and the energy on an arbitrary $\tv$ slice as
\begin{equation}\label{energy3}
 E[u](\tv_0) = \int_{ \M_R \cap \{\tv = \tv_0\}}
 \left(
|\partial_r u|^2 +   |\partial_\tv u|^2    + |\ang u|^2
\right) r^2  dr  d\omega
\end{equation}

The choice of the local energy norm $LE_S$ is inspired from
\eqref{lem}. However, there is a loss along  the trapped geodesics 
on the photon sphere. Consequently, we introduce  a
modified\footnote{notations are slightly changed compared to
  \cite{MMTT} in order to insure some uniformity across the 
three models described in the present paper}
$L^2$  local energy space
\begin{equation}\label{leS}
\| u\|_{LE_S} = \left\|  \left(1-\frac{3M}r \right) u\right\|_{LE_M}
\end{equation}
and $H^1$  local energy space
\begin{equation} \label{leS1}
\begin{split}
 \| u\|_{LE^1_S}  = \|\partial_r u\|_{LE_M} + \| \partial_{\tv}
 u\|_{LE_S} + \| \ang u\|_{LE_S} + \| r^{-1} u\|_{LE_M}
\end{split}
\end{equation}
For the inhomogeneous term we use the norm
\begin{equation} \label{leSstara}
\|f\|_{LE^*_{S}} = \left\|  \left(1-\frac{3M}r \right)^{-1} u\right\|_{LE_M}
\end{equation}
In the three formulas above we implicitely assume that all norms
are restricted to the set $\M_R$ where we study the wave equation
\eqref{boxsinhom}. Then we have the following result:
\begin{theorem}{\cite{MMTT}}\label{Schwarz}
 Let $u$ be so that $\Box_S u = f$.
 Then we have
\begin{equation}\label{main.est.Schw}
E[u](\Sigma_R^+) + \sup_{\tilde v} E[u](\tilde v) + \|u\|_{LE_S^1}^2
\lesssim E[u](\Sigma_R^-) + \|f\|_{LE^*_S}^2.
\end{equation}
\end{theorem}

Note that, compared to the norms $LE_M$, $LE_M^*$, the weights have 
an additional polynomial singularity at $r=3M$, but there are no
additional losses at the event horizon or near $\infty$. Furthermore,
by more refined results in \cite{MMTT}, this polynomial loss can be
relaxed to a logarithmic loss, i.e. the factor $1 - \frac{3M}r$ can be
improved to $|\ln (r-3M)|^{-1}$ near $r=3M$. This is related to the
fact that the (periodic) trapped rays on the photon sphere are
hyperbolic.

We also remark that in the expression of $LE_S^1$ it was sufficient to
measure $\partial_r u$. This is due to the implicit cancelation caused
by the fact that the symbol of the operator $\partial_r$ vanishes on
the trapped set.

The choice of $r_e \in (0,2M)$ is unimportant since the $r$-slices $r
= const \in (0,2M)$ are spacelike. Hence moving from one such
$r$-slice to another is equivalent to solving a local hyperbolic
problem, and involve no global considerations.  Thus in the proof of
the theorem one can assume without any restriction in generality that
$r_e$ is close to $2M$.

Local energy estimates were first proved in \cite{LS} for radially
symmetric Schr\"odinger equations on Schwarzschild backgrounds.  In
\cite{BS1, BS2, BSerrata}, those estimates are extended to allow for
general data for the wave equation.  The same authors, in
\cite{BS3,BS4}, have provided studies that give certain improved
estimates near the photon sphere $r=3M$.  Moreover, we note that
variants of these bounds have played an important role in the works
\cite{BSt} and \cite{DR}, \cite{DR2} which prove analogues of the
Morawetz conformal estimates on Schwarzschild backgrounds.

\newsection{Local energy decay in the Kerr space-time}

 The Kerr geometry in Boyer-Lindquist coordinates is given by
\[ 
ds^2 = g_{tt}dt^2 + g_{t\phi}dtd\phi + g_{rr}dr^2 + g_{\phi\phi}d\phi^2
 + g_{\theta\theta}d\theta^2 
\]
 where $t \in \R$, $r > 0$, $(\phi,\theta)$ are the spherical coordinates 
on $\S^2$ and 
\[
 g_{tt}=-\frac{\Delta-a^2\sin^2\theta}{\rho^2}, \qquad
 g_{t\phi}=-2a\frac{2Mr\sin^2\theta}{\rho^2}, \qquad
 g_{rr}=\frac{\rho^2}{\Delta}
 \]
\[ g_{\phi\phi}=\frac{(r^2+a^2)^2-a^2\Delta
\sin^2\theta}{\rho^2}\sin^2\theta, \qquad g_{\theta\theta}={\rho^2}
\]
with
\[ 
\Delta=r^2-2Mr+a^2, \qquad \rho^2=r^2+a^2\cos^2\theta. 
\]

 A straightforward computation gives us the inverse of the metric:
\[ g^{tt}=-\frac{(r^2+a^2)^2-a^2\Delta\sin^2\theta}{\rho^2\Delta},
\qquad g^{t\phi}=-a\frac{2Mr}{\rho^2\Delta}, \qquad
g^{rr}=\frac{\Delta}{\rho^2},
\]
\[ g^{\phi\phi}=\frac{\Delta-a^2\sin^2\theta}{\rho^2\Delta\sin^2\theta}
, \qquad g^{\theta\theta}=\frac{1}{\rho^2}.
\]

The case $a = 0$ corresponds to the Schwarzschild space-time.  We shall
subsequently assume that $a$ is small $a \ll M$, so that the Kerr
metric is a small perturbation of the Schwarzschild metric. We let
$\Box_K$ denote the d'Alembertian associated to the Kerr metric.

In the above coordinates the Kerr metric has singularities at $r = 0$
on the equator $\theta = \pi/2$ and at the roots of $\Delta$, namely
$r_{\pm}=M\pm\sqrt{M^2-a^2}$.  As in the case of the Schwarzschild
space, the singularity at $r=r_{+}$ is just a coordinate singularity,
and corresponds to the event horizon.  The singularity at $r = r_-$ is
also a coordinate singularity; for a further discussion of its nature,
which is not relevant for our results, we refer the reader to
\cite{Ch,HE}.  To remove the singularities at $r = r_{\pm}$ we
introduce functions $r^*$, $v_{+}$ and $\phi_{+}$ so that (see
\cite{HE})
\[
 dr^*=(r^2+a^2)\Delta^{-1}dr,
\qquad
 dv_{+}=dt+dr^*, 
\qquad 
 d\phi_{+}=d\phi+a\Delta^{-1}dr.
\]

The metric then becomes
\[
\begin{split}
ds^2= &\
-(1-\frac{2Mr}{\rho^2})dv_{+}^2+2drdv_{+}-4a\rho^{-2}Mr\sin^2\theta
dv_{+}d\phi_{+} -2a\sin^2\theta dr d\phi_{+} +\rho^2 d\theta^2 \\
& \ +\rho^{-2}[(r^2+a^2)^2-\Delta a^2\sin^2\theta]\sin^2\theta
d\phi_{+}^2
\end{split}
\]
which is smooth and nondegenerate across the event horizon up to but not including 
$r = 0$. Just like
in \cite{MMTT}, we introduce the function
\[
\tv = v_{+} - \mu(r)
\]
where $\mu$ is a smooth function of $r$. In the $(\tv,r,\phi_{+},
\theta)$ coordinates the metric has the form
\[
\begin{split}
ds^2= &\ (1-\frac{2Mr}{\rho^2}) d\tv^2
+2\left(1-(1-\frac{2Mr}{\rho^2})\mu'(r)\right) d\tv dr \\
 &\ -4a\rho^{-2}Mr\sin^2\theta d\tv d\phi_{+} + \Bigl(2 \mu'(r) -
 (1-\frac{2Mr}{\rho^2}) (\mu'(r))^2\Bigr)  dr^2 \\
 &\ -2a\theta (1+2\rho^{-2}Mr\mu' (r))\sin^2dr d\phi_{+} +\rho^2
 d\theta^2 \\
 &\ +\rho^{-2}[(r^2+a^2)^2-\Delta a^2\sin^2\theta]\sin^2\theta
d\phi_{+}^2.
\end{split}
\]

On the function $\mu$ we impose the following two conditions:

(i) $\mu (r) \geq  \rs$ for $r > 2M$, with equality for $r >
{5M}/2$.

(ii)  The surfaces $\tv = const$ are space-like, i.e.
\[
\mu'(r) > 0, \qquad 2 - (1-\frac{2Mr}{\rho^2}) \mu'(r) > 0.
\]
As long as $a$ is small, we can work with  the same 
function $\mu$ as in the case of the Schwarzschild space-time.

 For convenience we also introduce
\[
\tphi = \zeta(r)\phi_{+}+(1-\zeta(r))\phi
\]
where $\zeta$ is a cutoff function supported near the event horizon
and work in the $(\tv,r,\tphi, \theta)$ coordinates which are
identical to $(t,r,\phi,\theta)$ outside of a small neighborhood of
the event horizon.

Carter \cite{C} showed that the Hamiltonian flow is completely
integrable by finding a fourth constant of motion $K$ that is
preserved along geodesics. If $E$ and $L$ are the two constants of
motion associated with the Killing vector fields $\partial_t$ and
$\partial_\phi$, the equations for the null geodesics can
be reduced to the following (see, for example, \cite{Ch} or
\cite{Vol})
\begin{equation}
\begin{split}
 \rho^2 \dot{t} = &\ a(L-Ea\sin^2 \theta) + \frac{(r^2+a^2)((r^2+a^2)E-aL)}{\Delta}
\\
 \rho^2 \dot{\phi}= &\ \frac{L-Ea\sin^2 \theta}{\sin^2 \theta}+\frac{(r^2+a^2)aE-a^2 L}{\Delta}
\\
 \rho^4 \dot{\theta}^2=&\  K- \frac{(L-Ea\sin^2 \theta)^2}{\sin^2 \theta}
\\
 \rho^4\dot{r}^2 =&\  -K\Delta + ((r^2+a^2)E-aL)^2
\end{split}
\label{geoparam}\end{equation}
where the overdot denotes differentiation with respect to an affine
parameter $s$. This parametrization of the null geodesics is
nondegenerate away from the surfaces $r = r_{\pm}$.

Next we discuss the geometry of the trapped null geodesics.  The level
sets $r = r_0$ of $r$ are time-like for $r_0 > r_+$, null for $r =
r_+$ and space-like for $r_- < r_0 < r_+$. The latter implies that
there are no trapped null geodesics inside the region $\{r_- < r <
r_+\}$. On the null surfaces $r = r_{\pm}$, through each point there
is a unique null vector which is tangent and which generates a
trapped null geodesics.

To find the trapped null geodesics in the region $r > r_+$ it suffices
to consider the behavior of the fourth degree polynomial
\[
P(r) =  -K\Delta + ((r^2+a^2)E-aL)^2
\]
in the last equation in \eqref{geoparam}. At least one of the
parameters $E$, $K$ and $L$ should be nonzero, and the third equation
shows that $K \geq 0$ and that we cannot simultaneously have $E=K=0$.
Thus $P$ is always nondegenerate.  The key observation is that that
the simple zeroes of $P$ correspond to turning points in the last
equation, and only the double zeroes are steady states.  There are
several cases to consider.

a) If $E = 0$ then $K > 0$. Thus $P$ has at most one positive root,
where it changes sign from $+$ to $-$. This root is a right turning
point for the ode, and there are no trapped null geodesics.

b) $E \neq 0$. Then $P$ has degree $4$ and $P \geq 0$ in $[r_-,r_+]$.
If $P$ has any zero in $[r_-,r_+]$ then the square expression must
vanish, and this zero must be a double zero.  We claim that in $(r_+,
\infty)$ $P$ has either no root or two roots (counted with
multiplicity);  this is easily seen, as $P$ must have either (at least)
two complex conjugate roots or a negative root (the sum of the roots
equals $0$) and (at least) another one smaller than $r_-$ (since
$P(r_-)\geq 0$).  There are three subcases:

b1) $P$ has no roots larger than $r_+$. Then $r$ is monotone along
null geodesics, and there are no trapped null geodesics.

b2) $P$ has two distinct positive roots $ r_+ < r_1 < r_2$.  There it
must change sign from $+$ to $-$, respectively from $-$ to $+$.  Hence
$r_1$ is a right turning point and $r_2$ is a left turning point for
the ode. Thus  no trapped null geodesics exist.

b3) $P$ has a double positive real root $r_0$. Then this root is 
a steady state, and all other solutions converge to the steady   
state at one end, and escape to $0$ or infinity at the other end.

This analysis shows that the only trapped null geodesics are those 
along which $r$ is constant. The polynomial $P$ has a double 
root if the following two relations hold,
\[
((r^2+a^2)E-aL)^2 = K \Delta, \qquad 2rE ((r^2+a^2)E-aL) = K(r-M)
\]
which we rewrite in the form
\[
K = \frac{r^2 E^2 \Delta}{(r-M)^2},\qquad aL = E\left(r^2 + a^2   - \frac{2r
  \Delta}{r-M}\right)
\]
The right hand side in the $\dot{\theta}$ equation must be nonnegative.
Substituting in the above two relations we obtain a necessary
condition for the existence of trapped geodesics,
namely the inequality
\begin{equation}
 (2r\Delta - (r-M)\rho^2)^2\leq 4a^2 r^2\Delta\sin^2 \theta
\end{equation}
One can show that this condition is also sufficient.  The expression
on the left has the form
\[
2r\Delta - (r-M)\rho^2 = r^2(r-3M) + 2ra^2 - (r-M)a^2 \cos^2 \theta
\]
If $a=0$ then it has a single positive nondegenerate zero at $r=3M$,
which is the photon sphere in the Schwarzschild metric. Hence if $0 <
a \ll M$ it will still have a single zero which is close to $3M$. A 
rough computation leads to a bound of the form
\begin{equation}
|r-3M| \leq 2a, \qquad a \ll 2M
\label{closetoph}\end{equation}
Thus all trapped null geodesics lie within $O(a)$ of the 
$r = 3M$ sphere.

We would like a characterization of the aforementioned trapped
geodesics in the phase space.  Let $\tau, \xi, \Phi$ and $\Theta$ be
the Fourier variables corresponding to $t, r, \phi$ and $\theta$, and
\[
p( r, \phi,\tau, \xi, \Phi,\Theta)
=g^{tt}\tau^2+2g^{t\phi}\tau\Phi+g^{\phi\phi}\Phi^2
+g^{rr}\xi^2 +g^{\theta\theta}\Theta^2
\]
 be the principal symbol of $\Box_K$. On any null geodesic one has
\begin{equation}\label{Ham}
 p(t, r, \phi,\theta,\tau, \xi, \Phi,\Theta)=0.
\end{equation}
 
 Moreover, the Hamilton flow equations  give us
\begin{equation}\label{rdot}
 \dot{r}=-\frac{\partial p}{\partial \xi}=-\frac{2\Delta}{\rho^2} \xi
\end{equation}
\begin{equation}\label{xidot}
  \dot{\xi}=\frac{\partial p}{\partial r}= g^{tt}_{,r}\tau^2+2g^{t\phi}_{,r}\tau\Phi+g^{\phi\phi}_{,r}\Phi^2
  +g^{rr}_{,r}\xi^2 +g^{\theta\theta}_{,r}\Theta^2
\end{equation}
We rewrite the latter in the form
\begin{equation}\label{xidota}
  \rho^2 \dot{\xi}=\rho^2 \frac{\partial p}{\partial r}=
  - 2R_a(r,\tau,\Phi) \Delta^{-2} + \rho^2 \partial_r(\rho^{-2}) p + 2(r-M)\xi^2
\end{equation}
where 
\[
R_a(r,\tau,\Phi) = 
(r^2+a^2)(r^3-3Mr^2+a^2r+a^2M)\tau^2 - 2aM(r^2-a^2)\tau\Phi 
 - a^2(r-M)\Phi^2
\]
For geodesics with constant $r$, one needs to impose the additional
condition $\dot{r} = 0$. Hence from \eqref{rdot} either $r=r_{\pm}$,
which correspond to the geodesics at $r=2M$ in the Schwarzschild case,
or $\xi=0$. In the latter case  from \eqref{xidota} 
we obtain a polynomial equation for $r$, namely
\begin{equation}\label{Reqn}
 R_a(r,\tau,\Phi)= 0 
\end{equation}

Furthermore, due to \eqref{Ham} we must also have the inequality
\[
- ((r^2+a^2)^2-a^2\Delta\sin^2\theta)
\tau^2 - 2a Mr  \tau \Phi + \frac{\Delta-a^2\sin^2\theta}{\sin^2
  \theta} \Phi^2 \leq 0
\]
If $a$ is small and $r$ is as in \eqref{closetoph} this allows us to
bound $\Phi$ in terms of $\tau$,
\begin{equation}
|\Phi| \leq 4 M |\tau|
\end{equation}
For $\Phi$ in this range and small $a$ the polynomial 
$\tau^{-2} R_a(r,\tau,\Phi)$ can be viewed as a small perturbation of 
\[
\tau^{-2} R_0(r,\tau,\Phi) = r^4(r-3M) 
\]
which has a simple root at $r = 3M$. Hence for small $a$ the
polynomial $R_a$ has a simple root close to $3M$, which we denote by
$r_a(\tau,\phi)$. By homogeneity considerations and the implicit
function theorem we can further express $r_a$ in the form
\[
r_a(\tau,\Phi) = 3M \tilde r \left(\frac{a}{M}, \frac{\Phi}{M
    \tau}\right), \qquad \tilde r \in C^\infty(  [-\epsilon,\epsilon] \times [-4,4] )
\]
 Since $r_0(\tau,\Phi) = 3M$, it follows that
we can write $r_a(\tau,\Phi)$  in the form
\[
r_a(\tau,\Phi) = 3M + a F\left(\frac{a}{M}, \frac{\Phi}{M \tau}\right), 
 \qquad F \in C^\infty(  [-\epsilon,\epsilon] \times [-4,4] )
\]

The above analysis shows that the trapped null geodesics corresponding
to frequencies $(\tau,\Phi)$ are located at radial frequency $\xi = 0$
and position $r = r_a(\tau,\xi)$. One would be naively led to define
the local smoothing spaces associated to the Kerr space-time by
replacing the factor $r-3M$ in \eqref{leS} and \eqref{leSstara} with
the modified factor $r - r_a(\tau,\Phi)$. Unfortunately, this is no
longer a scalar function, but a symbol of a pseudodifferential
operator. In addition, this operator depends on the time Fourier
variable $\tau$, which is inconvenient for energy estimates on time
($\tv$) slabs.

Consequently, we replace the $r-r_a(\tau,\Phi)$ weight with a
polynomial in $\tau$ which has the same symbol on the characteristic
set $p = 0$. More precisely, for $r$ close to $3M$ we factor 
\[
p( r, \phi,\tau, \xi, \Phi,\Theta) = g^{tt}(\tau -\tau_1( r, \phi, \xi,
\Phi,\Theta)) (\tau -\tau_2( r, \phi, \xi,
\Phi,\Theta))
\]
where $\tau_1$, $\tau_2$ are real distinct smooth $1$-homogeneous symbols.
On the cone $\tau = \tau_i$ the symbol
$r-r_a(\tau,\phi)$ equals
\[
c_i ( r, \phi, \xi, \Phi,\Theta) = r- r_a(\tau_i,\Phi) = r-3M -a F\left(\frac{a}{M}, \frac{\Phi}{M \tau_i}\right), \qquad i = 1,2
\]
If $r$ is close to $3M$ and $|a| \ll M$ then on the characteristic set
of $p$ we have $|\phi| < 4M |\tau|$, therefore the symbols $c_i$ are
well defined, smooth and homogeneous. They are also nonzero outside 
an $O(a)$ neighborhood of $3M$. 

 We use the symbols $c_i$ to define associated microlocally weighted function spaces
$L^2_{c_i}$ in a neighborhood $I \times \S^2$ of $3M \times \S^2$
which does not depend on $a$ for small $a$. For 
functions $u$ supported in $I \times \S^2$ we set
\[
\| u\|_{L^2_{c_i}}^2 = \| c_i (D,x) u\|_{L^2}^2 + \|u\|_{H^{-1}}^2 
\]
There is an ambiguity in this notation as we have not specified the
coordinate frame in which we view $c_i$ as a pseudodifferential
operator. However, it is easy to see that different frames lead to
equivalent norms. We note that the low frequencies in $c_i$ are also
irrelevant, and can be removed with a suitable cutoff. After removing
the low frequencies, the quantization that we use for $c_i$ becomes
unimportant as well. We also define a dual norm $c_i L^2$ for 
functions $g$ supported in $I \times \S^2$, namely
\[
\| g\|_{c_i L^2}^2 = \inf_{c_i(x,D) g_1 + g_2 = g} (\|g_1\|_{L^2}^2 +
 \|g_2\|_{H^1}^2) 
\]

Since the symbols $c_i$ are nonzero outside an $O(a)$ neighborhood of 
$3M$, it follows that both norms $L^2_{c_i}$ and $c_i L^2$ are equivalent
to $L^2$ outside a similar neighborhood.

Now we can define local smoothing norms associated to the Kerr
space-time.  Let $\chi(r)$ be a smooth cutoff function which is
supported in the above neighborhood $I$ of $3M$ and which equals $1$ 
near $3M$. Then we set
\begin{equation}
\begin{split}
\|u\|_{LE_K^1} = & \|\chi(D_t - \tau_2(D,x))\chi u\|_{L^2_{c_1}} 
+ \|\chi (D_t - \tau_1(D,x))\chi u\|_{L^2_{c_2}} + \|(1-\chi^2) \partial_t  u\|_{LE_M}
\\ &\ +  \|(1-\chi^2) \ang  u\|_{LE_M}
+ \| \partial_r u\|_{LE_M} + \| r^{-1} u \|_{LE_M}
\end{split}
\label{leK}\end{equation}
We remark that this norm is equivalent to the Minkowski norm $LE_M^1$ 
outside an $O(a)$ neighborhood of $3M$, but it is degenerate on the trapped set.

For the nonhomogeneous term in the equation we define a dual structure,
\[
\|f\|_{LE_K^*} = \| (1-\chi) f\|_{LE_M^*} + \|\chi f\|_{c_1 L^2 + c_2 L^2}
\]

To state the main result of this paper we use the notations in
\eqref{mr}-\eqref{energy3}, with the parameter $r_e$ chosen so that
$r_- < r_e < r_+$:

\begin{theorem}\label{Kerr}
 Let $u$ solve $\Box_K u = f$ in $\M_R$. Then
\begin{equation}
 \|u \|_{LE_K^1}^2 + \sup_{\tilde v} E[u](\tilde v) + E[u](\Sigma_R^+)
 \lesssim E[u](\Sigma_R^-)+ \|f \|_{LE_K^*}^2.
\end{equation}
in the sense that the left hand side is finite and the inequality
holds whenever the right hand side is finite.
\end{theorem}

The proof of the result uses the multiplier method. Part of the
difficulty is caused by the fact that, as shown in \cite{Al}, there is
no differential multiplier that provides us with a positive local
energy norm.  What we do instead is find a suitable pseudodifferential
operator that does the job. This is chosen so that its symbol vanishes
on trapped rays, which leads to a local energy norm which is degenerate
there.

As in the Schwarzschild case, the choice of $r_e \in (r_-,r_+)$ is
unimportant since the $r$-slices $r = const \in (r_-,r_+)$ are
spacelike.  Hence in the proof of the theorem one can assume without
any restriction in generality that $r_e$ is close to $r_+$.
\begin{proof}

The theorem is proved using a modification of the arguments in
\cite{MMTT}.  Let us first quickly recall the key steps in the proof
of Theorem \ref{Schwarz} as in \cite{MMTT}.  We begin with the
energy-momentum tensor
\[
Q_{\alpha\beta}[u]=\partial_\alpha u \partial_\beta u -
\frac{1}{2}g_{\alpha\beta}\partial^\gamma u \partial_\gamma u
\]
Its contraction with respect to a vector field $X$ is denoted by
\[
P_\alpha[u,X]=Q_{\alpha\beta}[u]X^\beta
\]
and its divergence is
\[
\nabla^\alpha P_\alpha[u,X] = \Box_g u \cdot Xu +
\frac{1}{2}Q_{\alpha\beta}[u]\pi^{\alpha\beta}
\]
where $\pi^{\alpha \beta} $ is the deformation tensor of $X$,
given by
\[
\pi_{\alpha\beta}=\nabla_\alpha X_\beta + \nabla_\beta X_\alpha
\]
A special role is played by the Killing vector field 
\[
K=\partial_{\tv}
\]
whose deformation tensor is zero.

Integrating the above divergence relation for a suitable choice of $X$
does not suffice in order to prove the local energy estimates, as in
general the deformation tensor can only be made positive modulo a
Lagrangian term of the form $q \partial^\alpha u \partial_\alpha u $.
Hence some lower order corrections are required.  For a vector field
$X$, a scalar function $q$ and a 1-form $m$ we define
\[
P_\alpha[u,X,q,m] = P_\alpha[u,X] + q u \partial_\alpha u - \frac12
\partial_\alpha q u^2 + \frac{1}{2}m_{\alpha}u^2.
\]
The divergence formula gives 
\begin{equation}
\nabla^\alpha P_\alpha[u,X,q,m] =  \Box_g u \Bigl(Xu +
 q u\Bigr)+ Q[u,X,q,m],
\label{div}\end{equation}
where
\[
 Q[u,X,q,m] = 
\frac{1}{2}Q_{\alpha\beta}[u]\pi^{\alpha\beta} + q
\partial^\alpha u\, \partial_\alpha u + m_\alpha u\,
\partial^\alpha u + (\nabla^\alpha m_\alpha -\frac{1}{2}
\nabla^\alpha \partial_\alpha q) \, u^2.
\]
So far these computations apply both for the Schwarzschild and the Kerr metrics.
From here one, we will use the sub(super)scripts $S$, respectively $K$ 
to indicate when a computation is perform with respect to one metric or another.

To prove the local energy decay in the Schwarzschild space-time,
$X$, $q$ and $m$ are chosen as in the following lemma:

\begin{lemma}
There exist a smooth vector field
\[
 X=b(r)(1-\frac{3M}{r})\partial_r + c(r)K
\]
with $c$ supported near the event horizon and  $b>0$ bounded
so that  
\[
|\partial_r^\alpha b| \leq c_\alpha r^{-\alpha}
\] 
a smooth function $q(r)$ with
\[
|\partial_r^\alpha q| \leq c_\alpha r^{-1-\alpha}
\]
and a smooth $1$-form $m$ supported near the
event horizon $r=2M$ so that

(i) The quadratic form $Q^S[u,X,q,m]$
is positive definite,
\begin{equation}
Q^S[u,X,q,m] \gtrsim r^{-2} |\partial_r u|^2 + \left(1-\frac{3M}r
\right)^2 (r^{-2} |\partial_\tv u|^2 + r^{-1}|\ang u|^2) + r^{-4}
u^2.
\label{posS}\end{equation}

(ii) $X(2M)$ points toward the black hole, $X(dr)(2M) < 0$, and
$\langle m,dr\rangle(2M) > 0$.

\label{ibp}
\end{lemma}

The local energy estimate is obtained by integrating the divergence
relation \eqref{div} with $X+CK$ instead of $X$, where $C$ is a large
constant,  on the domain 
\[
D = \{ 0 < \tv < \tv_0,\ r > r_e \}
\]
with respect to the measure induced by the metric, $dV_S = r^2 dr d\tv
d\omega$.  This yields
\begin{equation}
\int_D Q^S[u,X,q,m]  dV_S = 
- \int_D \Box_S u \Bigl((X+CK)u + q u\Bigr) dV_S
+ BDR^S[u]
\label{intdiv}\end{equation}
where $BDR^S[u]$ denotes the boundary terms
\[
BDR^S[u] = \left. \int \langle d\tv, P[u,X + C K,q,m]\rangle
  r^2 dr  d\omega \right|_{\tv = 0}^{\tv = \tv_0}
 \!\!  - \!  \int \langle dr,  P[u,X + C K,q,m]\rangle
  r_e^2 d\tv  d\omega 
\]
Using the condition (ii) in the Lemma and Hardy type inequalities,
it is shown in \cite{MMTT} that for large  $C$ and $r_e$ close 
to $2M$ the boundary terms have the correct sign,
\begin{equation}
  BDR^S[u]  \leq c_1 E[u](\Sigma_R^-) - 
c_2 (E[u](\tv_0) + E[u](\Sigma_R^+)), \qquad c_1, c_2 > 0
\label{bdrpos}\end{equation}
Consequently, by applying the Cauchy-Schwartz inequality 
for the first term on the right of \eqref{intdiv} we obtain 
a slightly weaker form of the local energy estimate 
\eqref{main.est.Schw}, namely 
\begin{equation}\label{main.est.Schwa}
E[u](\Sigma_R^+) + \sup_{\tilde v} E[u](\tilde v) + \|u\|_{LEW_S^1}^2
\lesssim E[u](\Sigma_R^-) + \|f\|_{LEW^*_S}^2.
\end{equation}
where the weaker norm $LEW_S^1$ and the stronger norm $LEW^*_S$
are defined by 
\[
\| u\|_{LEW_S^1}^2 = \int_{\M_R} \left(
 r^{-2} |\partial_r u|^2 + \left(1-\frac{3M}r
\right)^2 (r^{-2} |\partial_\tv u|^2 + r^{-1}|\ang u|^2) + r^{-4}
u^2 \right) r^2 dr d\tv d\omega
\]
respectively
\[
\| f\|_{LEW_S^*}^2 =  \int_{\M_R}  r^2  \left(1-\frac{3M}r
\right)^{-2} f^2 r^2 dr d\tv d\omega
\]
These norms are equivalent with the stronger norms $LE_S^1$,
respectively $LE^*_S$ for $r$ in a bounded set. On the other hand for
large $r$ the Schwarzschild space can be viewed as a small perturbation
of the Minkovski space. Thus the transition from
\eqref{main.est.Schwa} to \eqref{main.est.Schw} is achieved in
\cite{MMTT} by cutting away a bounded region and then using a
perturbation of a Minkowski space estimate. This part of the proof
translates without any changes to the case of the Kerr space-time.
Our goal in what follows will be to establish the Kerr counterpart
of \eqref{main.est.Schwa}, namely 
\begin{equation}\label{main.est.Kerr}
E[u](\Sigma_R^+) + \sup_{\tilde v} E[u](\tilde v) + \|u\|_{LEW_K^1}^2
\lesssim E[u](\Sigma_R^-) + \|f\|_{LEW^*_K}^2.
\end{equation}
where the norms $LEW_K^1$, respectively $LEW^*_K$ coincide
with $LE_K^1$, respectively $LE^*_K$ for bounded $r$, and with
$LEW_S^1$, respectively $LEW^*_S$ for large $r$. More precisely,
if $\chi(r)$ is a smooth compactly supported cutoff function which
equals $1$ say for $r < 4M$ then we set
\[
\| u\|_{LEW_K^1}^2 =
 \| \chi u\|_{LE_K^1}^2 + \|(1-\chi) u\|_{LEW_S^1}^2 
\]
respectively
\[
\| u\|_{LEW_K^*}^2 =
 \| \chi u\|_{LE_K^*}^2 + \|(1-\chi) u\|_{LEW_S^*}^2 
\]
Different choices for $\chi$ lead to different but equivalent norms.

It is useful to first consider the effect of the same multiplier in
the Kerr metric. The two metrics are close when measured in the 
same euclidean frame $x = r \omega$ with $r \geq r_e$. Precisely, with
$\partial$ standing for $\partial_t$ and $\partial_x$, $x = r \omega$,

\begin{equation}\label{aproxmet}
 |\partial^\alpha [(g_K)_{ij}-(g_{S})_{ij}]|\leq c_\alpha
 \frac{a}{r^{2+|\alpha|}}, \qquad 
 |\partial^\alpha [(g_K)^{ij}-(g_{S})^{ij}]|\leq c_\alpha
 \frac{a}{r^{2+|\alpha|}}
\end{equation}
From this and the size and regularity properties of $X$, $q$ and $m$
it follows that 
\begin{equation}
|P_\alpha^S[u,X,q,m] -P_\alpha^K[u,X,q,m]| \lesssim \frac{a}{r^2}
|\nabla u|^2 
\label{pdiff}\end{equation}
respectively 
\begin{equation}
|Q^S[u,X,q,m] -Q^K[u,X,q,m]| \lesssim a \left( \frac{1}{r^2}
|\nabla u|^2 + \frac{1}{r^4} |u|^2\right)
\label{qdiff}\end{equation}
Hence integrating the divergence relation \eqref{div} in the Kerr
space-time over the same domain $D$  but with respect to the Kerr
induced measure $dV_K=\rho^2 dr d\tv d\omega$ we obtain
\begin{equation}
\int_D Q^K[u,X,q,m] dV_K = 
- \int_D \Box_K u \Bigl((X+CK)u + q u\Bigr) dV_K
+ BDR^K[u]
\label{intdivk}\end{equation}
The bound \eqref{pdiff} shows that for small $a$ the boundary terms
retain their positivity properties in \eqref{bdrpos}, namely
\begin{equation}
  BDR^K[u]  \leq c_1 E[u](\Sigma_R^-) - 
c_2 (E[u](\tv_0) + E[u](\Sigma_R^+)), \qquad c_1, c_2 > 0
\label{bdrposk}\end{equation}
However, \eqref{qdiff} merely shows that 
\begin{equation}
Q_K[u,X,q,m] \gtrsim r^{-2} |\partial_r u|^2 + \left[\left(1-\frac{3M}r
\right)^2-Ca\right] (r^{-2} |\partial_\tv u|^2 + r^{-1}|\ang u|^2) + r^{-4}
u^2
\label{qkbd}\end{equation}
and the right hand side is no longer positive definite near $r = 3M$.
Thus we cannot close the argument as in the Schwarzschild case.  As
shown in \cite{Al}, changing the vector field $X$ near $r = 3M$ would
not help.

To remedy this, we need to use a pseudodifferential modification $S$
of the vector field $X$. We will choose $S$ so that its kernel is
supported in a small neighborhood of $(3M,3M)$; this insures that
there will be no additional contributions at $r = r_e$.  Furthermore,
in order to be able to carry out the computations near the initial and
final surfaces $\tv = 0,\tv_0$ we take $S$ to be a first order differential
operator with respect to $\tv$.  Similarly, we modify the Lagrangian 
factor $q$ using a pseudodifferential correction $E$, which is also
a first order differential operator with respect to $\tv$. 

We also need to choose a quantization which is consistent with the
Kerr measure. Here we have a few choices which have equivalent
results. For our selection we use euclidean-like coordinates $x =
\omega r$.  Given a real symbol $s$ its euclidean Weyl quantization
$s^w$ is selfadjoint with respect to the euclidean measure $dV = r^2
dr d\omega$. However, in our case we need to work with the Kerr
induced measure $dV_K = \rho^2 dr d\omega$. Hence we slightly abuse
notations and redefine the Weyl quantization as 
\[
s^w:= \frac{r}{\rho} \ s^w\  \frac{\rho}{r}
\]
If $s$ is a real symbol, then $s^w$ (re)defined above is a selfadjoint
operator in $L^2(dV_K)$. 

Another issue which does not affect our analysis but needs to be
addressed is that we are using pseudodifferential operators in an
exterior domain $\{ r > r_e\}$ and some care must be given to what 
happens near $r = r_e$. To keep things simple, in what follows 
 all operators we work with are compactly supported in the sense 
that their kernels are supported away from $r_e$ and infinity; even
better, supported in a small neighborhood of $3M$.

 In what follows we consider a
skewadjoint pseudodifferential operator $S$ and a selfadjoint
pseudodifferential operator $E$ of the form \footnote{Since we are away from the 
event horizon the variable $\tv$ coincides with $t$. We make this substitution here and later.}
\[
S =  i s_1^w + s_0^w \partial_t, \qquad E =  e_0^w 
+ \frac{1}i  e_{-1} \partial_t
\]
where $s_1 \in S^1$, $s_0,e_0 \in S^0$ and $e_{-1} \in S^{-1}$ are
real symbols, homogeneous outside a neighborhood of $0$.  Commuting
and integrating by parts we obtain the counterpart of the relation
\eqref{intdivk}, namely
\begin{equation}
 IQ^K[u,S,E]  = - \Re \int_D \Box_K u \cdot
(S+E) u dV_K + BDR^K[u,S,E]
\label{intdivs}\end{equation}
Here $BDR^K[u,S,E]$ represents the boundary terms obtained in the 
integration by parts with respect to $t$. Its exact form is not
important, as all we need to use here is the bound 
\begin{equation}
|BDR^K[u,S,E]| \lesssim E[u](0) + E[u](\tv_0)
\label{bdrks}
\end{equation}
We note that due to the presence of the cutoff function $\chi$ in both
operators $S$ and $E$ there are no contributions on the lateral
boundary $r = r_e$.

On the other hand $IQ^K[u,S,E]$ represents a quadratic form
in $(u,D_t u)$ which can be written in the form
\begin{equation}
IQ^K[u,S,E]  = \int_D Q_{2}^w u \cdot \overline u + 2\Re
Q_1^w u \cdot \overline {D_t u} + Q_0^w {D_t u} \overline {D_t u} \
dV_K
\label{IQrep}\end{equation}
where $Q_j^w \in OPS^j$ are selfadjoint pseudodifferential operators
so that
\begin{equation}
Q_2^w  + 2 Q_1^w D_t + Q_0^wD_t^2 = 
\frac12([\Box_K,S] + \Box_K E + E \Box_K)
\label{qjdef}\end{equation}
We remark that for arbitrary operators $S$ and $E$ the 
expression above on the right is in general a third order 
differential operator in $t$. However, the operator $S$ will 
always be chosen so that the coefficient of $D_t^3$ vanishes.
We define the principal symbol of the quadratic form
$IQ^K[u,S,E]$ as 
\[
q^K[S,E] =  q_2 + 2 q_1 \tau + q_0 \tau^2
\]
The previous relation shows that it satisfies
\[
q^K[S,E] = \frac{1}{2i} \{p,s\} + p e  \qquad
\text{mod } S^0 + \tau S^{-1} +\tau^2 S^{-2}
\]

We add \eqref{intdivk} with $a$ times \eqref{intdivs}. The boundary 
terms are estimated by \eqref{bdrposk} and \eqref{bdrks}. Using the
duality between the spaces $c_i L^2$ and $L^2_{c_i}$ we can also
estimate
\[
\left| \int_D f  (X + CK + q + a(S+E)) u dV_K \right|
\lesssim \|f\|_{LEW^*_K} \| u\|_{LEW^1_K} 
\]
Hence in order to prove \eqref{main.est.Kerr} it would suffice to show
that the symbols $s$ and $e$ can be chosen so that
\begin{equation}
  \int_D  Q^K[u,X,q,m]  dV_K   + a  IQ^K[u,S,E] 
\gtrsim \|u\|_{LEW_K^1}^2
\label{mek}\end{equation}
Here we aim to choose $S$ and $E$ uniformly with respect to 
small $a$. In effect, our construction below yields symbols $s$ and
$e$ which are analytic with respect to $a$. We remark that the choice 
of $S$ and $E$ is only important in the region where $r$ is close to
$3M$. Outside this region, $Q^K[u,X,q,m]$ is already
positive definite and the contribution of $ a  IQ^K[u,S,E] $ is
negligible. 

We consider first the expression $Q^K[u,X,q,m]$. Near $r=3M$ this has
the form
\begin{equation}
Q^K[u,X,q,m]  = \sum q^{K,\alpha\beta} \partial_\alpha u
\partial_\beta u  + q^{K,0} u^2
\label{qkform}\end{equation}
where its principal symbol $q^K = q^{K,\alpha\beta} \eta_\alpha
\eta_\beta$ and the lower order coefficient $q^{K,0}$ are given by the
relation
\[
q^K = \frac{1}{2i} \{p,X\} + q p, \qquad q^{K,0} = -  \frac12
\nabla^\alpha \partial_\alpha q
\]
We do not need to exactly compute the above expression in the Kerr
case, but it is useful to perform the computation in the simpler case
of the Schwarzschild space. There we have
\[
p = - \left(1-\frac{2M}r \right)^{-1} \tau^2 + \left(1-\frac{2M}r
\right) \xi^2 + \frac{1}{r^2} \lambda^2, \qquad X = i b(r)
\left(1-\frac{3M}r\right) \xi
\]
where $\lambda$ stands for the spherical Fourier variable.
Hence we obtain
\begin{equation}
\begin{split}
r^2 q^S =&\   \frac{1}{2i} \{r^2 p,X\} +( q - r^{-1} b(r)(r-3M)) (r^2 p) 
\\ =&\  \alpha_S^2(r)  \tau^2 + \beta_S^2(r)  \xi^2 + \tilde q(r)(r^2 p)   
\end{split}
\label{sqss}\end{equation}
where, near  $r = 3M$,
\[
\alpha_S^2(r) =  \frac{ r b(r) (r-3M)^2}{(r-2M)^2},
\]
\[
\quad \beta_S^2(r) = \frac{3M}{r^2} b(r^2 -2Mr) + \left(1-\frac{3M}r\right)(b'(r^2 -2Mr) -b(r-M))
\]
respectively 
\[
\tilde q (r) =  q - r^{-1} b(r)(r-3M).
\]
Here we have used the fact that $b > 0$ to write the first two
coefficients as squares.

For our choice of $q$ and $r$ we know that the relation 
\eqref{posS} holds. This implies that the following 
two inequalities must hold:
\begin{equation}
q^S \gtrsim \xi^2 + (r-3M)^2 (\tau^2 + \lambda^2), \qquad q^{S,0} > 0
\end{equation}
Given the form of $q^S$, the  first relation implies that 
$\tilde q$ is a multiple of $(r-3M)^2$, and that in addition 
there is a smooth function $\nu(r)$ so that
\[
\frac{r^3}{r-2M}  \tilde q = \nu(r) \alpha_S^2(r), \qquad 0 < \nu < 1
\]
This allows us to obtain the following sum of squares representation
for $q^S$:
\begin{equation}
  r^2 q^S =  (1-\nu(r)) \alpha_S^2(r)  \tau^2 + \beta^2(r)  \xi^2 +
  \nu_1(r)\alpha_S^2(r) (\lambda^2
  + (r^2 -2rM) \xi^2), \quad \nu_1 = \frac{r-2M}{r^3} \nu 
\label{sumsqS0}\end{equation}
The symbol $\lambda^2$ of the spherical Laplacian ca also be written
as sums of squares of differential symbols,
\[
\lambda^2 = \lambda_1^2 + \lambda_2^2 + \lambda_3^2 
\]
where in Euclidean coordinates we can write
\[
\{  \lambda_1, \lambda_2,\lambda_3\} = \{ x_i \eta_j -x_j \eta_i, \ i \neq
j\}
\]
\begin{equation}
  r^2 q^S =  (1-\nu(r)) \alpha_S^2(r)  \tau^2 + \beta_S^2(r)  \xi^2 +
  \nu_1(r) \alpha_S^2(r) (\lambda_1^2 +\lambda_2^2 +\lambda_3^2 
  + (r^2 -2rM) \xi^2)
\label{sumsqS}\end{equation}

We return now to the question of finding symbols $s$ and $e$ 
so that the bound \eqref{mek} holds. Near $r = 3M$, the 
principal symbol of the quadratic form on the left in \eqref{mek}
is 
\[
\frac{1}{2i} \{ p, X+as\} + p(q+a e)
\]
In order to prove \eqref{mek} at the very least we would like the
above symbol to be nonnegative, and to satisfy the bound
\[
\frac{1}{2i} \{ p, X+a s \} + p(q+a e)
 \gtrsim c_2^2(\tau -\tau_1)^2 + 
c_1^2(\tau-\tau_2)^2 + \xi^2
\]
However, such a bound would not a-priori suffice since translating it 
to operator bounds would require using the Fefferman-Phong inequality,
which does not hold in general for systems. Hence we prove 
a more precise result, and show that the symbols $s$ and $e$ can be
chosen so that we have a favorable sum of squares representation for
the above expression, which extends the sum of squares 
\eqref{sumsqS} to $a \neq 0$.

\begin{lemma} \label{squaresum}
  Let $a$ be sufficiently small. Then there exist smooth homogeneous
  symbols $s \in S^1_{hom}$, $e \in S^0_{hom}$, also depending
  smoothly in $a$, so that for $r $ close to $3M$ we have sum of
  squares representation
\begin{equation}
\rho^2\left( \frac1{2i} \{ p, X+a s \} + p(q+a e)\right) = \sum_{j = 1}^8 \mu_j^2 
\label{sumsqK}\end{equation}
where $\mu_j \in S^1_{hom}+\tau S^0_{hom}$ 
satisfy the following properties:

(i) The decomposition \eqref{sumsqK} extends the decomposition
\eqref{sumsqS} in the sense that
\[
\begin{split}
(\mu_1,\mu_2,\mu_3,\mu_4,\mu_5,\mu_6) = &\  
((1-\nu)^\frac12 \alpha_S \tau, \beta_S \xi, \nu_1^\frac12  \alpha_S\lambda_1,
 \nu_1^\frac12  \alpha_S\lambda_2, \nu_1^\frac12  \alpha_S\lambda_3,
 \nu_1^\frac12  \alpha_S\xi) \\ &\  \text{mod }   a(S^1_{hom}+\tau
 S^0_{hom}) 
\end{split}
\]
and
\[
(\mu_7,\mu_8) \in \sqrt{a} (S^1_{hom}+\tau S^0_{hom})
\]

(ii) The family of symbols $\{ \mu_j\}_{j=1,6}$ is elliptically
equivalent with the family of symbols
$(c_2(\tau-\tau_1),c_1(\tau-\tau_2), \xi)$ in the sense that we have a
representation of the form
\[
\mu = Mv, \qquad 
v = \left( \begin{array}{c}
c_2(\tau-\tau_1) \cr c_1(\tau-\tau_2) \cr \xi \end{array} \right) 
\]
where the symbol valued matrix $M \in M^{8\times 3}(S^{0}_{hom})$
has maximum rank $3$ everywhere.
\end{lemma}

\begin{proof}
  Setting $\tilde q = q - 2 \{ \ln \rho,X\}$ respectively $\tilde e =
  e - 2 \{ \ln \rho,s\}$ we compute
\[
\rho^2\left( \frac1{2i} \{ p, X+a s \} + (q+a e) p \right) = 
\frac1{2i} \{ \rho^2 p, X+a s \} + (\tilde q+a \tilde e)(\rho^2 p)
\]
We first choose the symbol $s$ so that the Poisson bracket $\{\rho^2
p, X+a s \}$ has the correct behavior on the characteristic set $p =
0$.  Recall that the symbol of $X$ is $i r^{-1} b(r) (r-3M) \xi$,
where the vanishing coefficient at $3M$ corresponds exactly to the
location of the trapped rays. Its natural counterpart in the Kerr
space-time is the symbol
\[
 \ts(r,\tau,\xi,\Phi)=i r^{-1}b(r)(r-r_0(\tau,\Phi))\xi.
\]
This coincides with $X$ in the Schwarzschild case $a=0$, and it is well
defined and smooth in $a$ for $r$ near $3M$ and $|\Phi| < 4|\tau|$. In
particular it is well defined in a neighborhood of the characteristic
set $p =0$, which is all we use in the sequel.

We use \eqref{xidota} to compute the Poisson bracket
\[
\begin{split}
\frac{1}{i}  \{ \rho^2 p,\ts\} =&\ -  (\rho^2 p)_r r^{-1} b(r)(r-r_0(\tau,\Phi))+  
\xi (\rho^2 p)_\xi \partial_{r}\left(r^{-1} b(r)(r-r_0(\tau,\Phi))\right) 
\\
 = &\ 2r^{-1} b(r) R(r,\tau,\Phi)\Delta^{-2}(r-r_0(\tau,\Phi))
\\ &\ +
 \left[2\Delta\partial_{r}\left(r^{-1} b(r)(r-r_0(\tau,\Phi))\right) - 2(r-M)
  r^{-1} b(r) (r-r_0(\tau,\Phi))  \right] \xi^2 
\end{split}
\]
Since $r_0(\tau,\Phi)$ is the unique zero of $R(r,\tau,\Phi)$ near
$r=3M$ and is close to $3M$, it follows that we can write
\begin{equation}
\frac{1}{2i} \{ \rho^2 p,\ts\} = 
\alpha^2(r,\tau,\Phi) \tau^2 (r-r_0(\tau,\Phi))^2 +
\beta^2(r,\tau,\Phi)\xi^2  \qquad \text{on} \ \ \{p = 0\}
\label{pbs}\end{equation}
where $\alpha, \beta \in S^0_{hom} $ are positive
symbols. We note that in the Schwarzschild case the symbols
$\alpha$ and  $\beta$ are simply functions of $r$, see 
the first two terms in \eqref{sqss}.

  Unfortunately $\ts$ is not a polynomial in $\tau$, which limits
its direct usefulness. To remedy that we first note that 
\[
\ts - (i r^{-1} b(r) (r-3M) \xi) \in a S^1_{hom}
\]
Hence by (the simplest form of)
the Malgrange preparation theorem we can write
\[
\frac{1}i \ts = (r-3M)b(r)\xi  +  a (s_1(r,\xi, \theta,\Theta,\Phi)+
s_0(r,\xi, \theta,\Theta,\Phi) \tau)+  
a  \gamma(\tau, r,\xi, \theta,\Theta,\Phi) p
\]
with $s_1 \in S^1_{hom}$, $s_0 \in S^0_{hom}$ and $\gamma \in
S^{-1}_{hom}$.  Then we define the desired symbol $s$ by
\[
s = s_1 + s_0 \tau
\]
The Poisson bracket $\frac{1}i \{\rho^2 p, s\}$ is a third degree polynomial
in $\tau$. Hence, after division by $p =
-g^{tt}(\tau-\tau_1)(\tau-\tau_2)$, taking also \eqref{sumsqS} into
account, we can write
\begin{equation}
\frac1{2i} \{ \rho^2 p, X+as\} + \tilde q (\rho^2 p)= 
\gamma_2 + \gamma_1 \tau + [e_S+ a(e_0 + e_{-1}
\tau)] (\tau-\tau_1)(\tau-\tau_2).
\label{divbyp}\end{equation}
where, by   \eqref{sumsqS}, the coefficient 
 $e_S$ corresponding  to the Schwarzschild case is given by 
\[
e_S = (1-\nu(r)) \alpha^2(r).
\]

It remains to show that the right hand side of \eqref{divbyp} can be
expressed as a sum of squares as in the lemma modulo an error
$a(S^0_{hom} + \tau S^{-1}_{hom}) p$.  Note that the symbols $e_0$ and
$e_{-1}$ play no role in this, as they can be included in the error.

The coefficients $\gamma_1$ and $\gamma_2$ can be computed using the
relation \eqref{pbs} and the fact that $\{\rho^2 p, X+as\}=\{\rho^2 p, \ts\}$ on
$p=0$ (i.e.  when $\tau=\tau_i$). We denote
\[
\alpha_{i} =\frac{2|\tau_{i}|}{\tau_{1} -\tau_{2}}
 \alpha(r,\tau_{i},\Phi) (r - r_0(\tau_{i},\Phi)) \in S^0_{hom},
\qquad \beta_{i} = \beta(r,\tau_{i},\Phi),
\] 
observing that $\alpha_i$ can be used as substitutes for the $c_i$'s
in the lemma since they are  elliptic multiples of $c_i$.
Then we have
\[
\frac{1}{2i} \{\rho^2 p, \ts\}(\tau_i) = \frac14 
\alpha^2_i (\tau_1-\tau_2)^2
+ \beta_i^2 \xi^2
\]
which gives the following expressions for $\gamma_1$, $\gamma_2$:
\begin{equation}
 \gamma_2=\frac14(\tau_1 -\tau_2) (\alpha_2^2 \tau_1- \alpha_{1}^2\tau_2  )
 + \frac{\tau_1 \beta_2^2 -\tau_2 \beta_1^2}{\tau_1-\tau_2} \xi^2,
\quad
 \gamma_1=\frac14(\tau_1-\tau_2)(\alpha_{1}^2 -
\alpha_{2}^2)
+ \frac{\beta_{1}^2 -\beta_{2}^2}{\tau_1-\tau_2} \xi^2
\label{gamma}\end{equation}

We use the first components of $\gamma_1$ and $\gamma_2$ 
to obtain a sum of squares as follows:
\begin{equation}
\begin{split}
\!\!\!\!(\tau_1 -\tau_2) (\alpha_2^2 \tau_1- \alpha_{1}^2\tau_2  ) + \tau (\tau_1-\tau_2)(\alpha_{1}^2 -
\alpha_{2}^2) &= 
\nu (\alpha_1 (\tau-\tau_2) -\alpha_2(\tau-\tau_1))^2 \\ &\ +
(1-\nu) (\alpha_1 (\tau-\tau_2) +\alpha_2(\tau-\tau_1))^2 
\\ & \ -4 e_K (\tau-\tau_1)(\tau-\tau_2) 
\end{split}
\label{firstsq}\end{equation}
where
\[
 e_K =  \frac{(\alpha_1-\alpha_2)^2}4 + (1-\nu) \alpha_1 \alpha_2
\]
We remark that in  the Schwarzschild case we have  $\tau_2 =
-\tau_1$  and also  $\alpha_1 = \alpha_2=\alpha_S$ and
$\beta_1=\beta_2=\beta_S$. In particular this shows 
that
\[
e_K -e_S \in a(S^0_{hom} + \tau S^{-1}_{hom})
\]
which accounts for the $e_S$ factor in \eqref{divbyp}.  It remains to
consider the $\xi^2$ terms in \eqref{gamma}. This is easier since the
coefficients $\beta_1$, $\beta_2$ are positive and have a small
difference $\beta_1-\beta_2 \in a S^{0}_{hom}$.  Precisely, for a
large $C$ we can write
\[
\begin{split}
\frac{\tau_1 \beta_2^2 -\tau_2 \beta_1^2}{\tau_1-\tau_2}  + 
\tau  \frac{\beta_{1}^2 -\beta_{2}^2}{\tau_1-\tau_2}  = &\ 
\frac12 (\beta_1^2 + \beta_2^2 - Ca) +  
\frac{(Ca - \beta_2^2+\beta_1^2) (\tau-\tau_2)^2}
{2(\tau_1-\tau_2)^2}
\\ &\ +\frac{ (Ca-\beta_1^2+\beta_2^2) (\tau-\tau_1)^2}{2(\tau_1-\tau_2)^2}
+ O(a) p
\end{split}
\]
Summing this with \eqref{firstsq} we obtain the desired sums of
squares representation,
\[
\begin{split}
 \frac{1}{2 i} \{\rho^2 p, X+a s \} + (\rho^2 p) \tilde q \in &\ \frac{\nu}4 (\alpha_1 (\tau-\tau_2)
  -\alpha_2(\tau-\tau_1))^2  \\ &\  + \frac{1-\nu}4 (\alpha_1 (\tau-\tau_2)
  +\alpha_2(\tau-\tau_1))^2 + \frac12 (\beta_1^2 + \beta_2^2 - Ca) \xi^2
 \\ &\ + \frac{(Ca - \beta_2^2+\beta_1^2) (\tau-\tau_2)^2} {2(\tau_1-\tau_2)^2}
  \xi^2  +\frac{ (Ca-\beta_1^2+\beta_2^2)
    (\tau-\tau_1)^2}{2(\tau_1-\tau_2)^2} \xi^2\\ &\ +
  a (S^0_{hom} + S^{-1}_{hom} \tau) (\tau -\tau_1)(\tau-\tau_2)
\end{split}
\]
Then $e$ is chosen so that the last term accounts for the contribution
of $\tilde e$.  

Part (ii) of the lemma directly follows. For part (i)
we still need to specify which are the symbols $\mu_j$. Precisely, we
set
\[
\mu_1^2 = \frac{1-\nu}4 (\alpha_1 (\tau-\tau_2)
  +\alpha_2(\tau-\tau_1))^2, \qquad \mu_2^2 = \frac12 (\beta_1^2 + \beta_2^2 - Ca) \xi^2
\]
\[
\mu_7^2 =  \frac{(Ca - \beta_2^2+\beta_1^2) (\tau-\tau_2)^2} {2(\tau_1-\tau_2)^2}
  \xi^2, \qquad \mu_8^2 = \frac{ (Ca-\beta_1^2+\beta_2^2)
    (\tau-\tau_1)^2}{2(\tau_1-\tau_2)^2} \xi^2 
\]
Finally for $\mu_{3,4,5}$ and $\mu_6$ we set
\[
\mu_{3,4,5}^2 =  \frac{\lambda_{1,2,3}^2}{\lambda^2+(r^2-2rM) \xi^2} \frac{\nu}4 (\alpha_1 (\tau-\tau_2)
  -\alpha_2(\tau-\tau_1))^2, 
\]
respectively
\[
 \mu_6^2 =
  \frac{(r^2-2rM) \xi^2}{\lambda^2+(r^2-2rM) \xi^2}
  \frac{\nu}4 (\alpha_1 (\tau-\tau_2)
  -\alpha_2(\tau-\tau_1))^2
\]
It is easy to see that for $a = 0$ these symbols coincide with the
coresponding Schwarzschild symbols. The proof of the lemma is
concluded.  

\end{proof}

In what follows we use the above lemma to prove the bound \eqref{mek}
and conclude the proof of the theorem. We begin with symbols $s$ and
$e$ as in the lemma. These are homogeneous symbols, and we can make
them smooth by truncating away the low frequencies.  They are only
defined near $r = 3M$, therefore some spatial truncation is also
necessary.  Let $\chi$ be a smooth cutoff function supported near $3M$
which equals $1$ in a neighborhood of $3M$, chosen so that we have a
smooth partition of unity in $r$,
\[
1 = \chi^2(r) + \chi_{o}^2(r)
\]
At first we define the truncated operators 
\[
\tilde S = \chi s^w \chi, \qquad \tilde E =
\chi e^w \chi
\]
This choice would yield an expression $Q^K[u,\tilde S,\tilde E]$ 
with a principal symbol
\[
q^K_{princ}[\tilde S,\tilde E] = \chi^2 \left( \frac{1}{2i} \{p,s\} + pe
\right) + \frac{1}{i} \chi s \{p,\chi\}
\]

For these choices of $\tilde S$ and $\tilde E$ we consider the
expression $IQ^K[u,\tilde S,\tilde E]$ which is given by \eqref{IQrep}
with $Q_2^w$, $Q_1^w$ and $Q^w_0$ as in \eqref{qjdef}.  We observe
that in general we can only say that the right hand side of
\eqref{qjdef} is of the form
\[
\frac12([\Box_K,\tilde S] + \Box_K \tilde E + \tilde E \Box_K) = 
Q_2^w  + 2 Q_1^w D_t + Q_0^wD_t^2 + Q_{-1}^w D_t^3, \qquad Q_j^w \in 
OPS^j
\]
However, its principal symbol $q^K_{princ}[\tilde S,\tilde E]$  is at most a second 
order polynomial in $\tau$. Hence by the Weyl calculus we can write 
\[
\frac12([\Box_K,\tilde S] + \Box_K \tilde E + \tilde E \Box_K)
- (q^K_{princ}[\tilde S,\tilde E])^w \in  OPS^0 + OPS^{-1}
D_t + OPS^{-2} D_t^2 + OPS^{-3} D_t^3
\]
In particular this shows that $ Q_{-1}^w \in  OPS^{-3}$. 
To eliminate this term we slightly adjust our choice of $\tilde E$
to
\[
\tilde E = \chi e^w \chi - e^w_{aux} D_t
\]
where the operator $e^w_{aux}$ is chosen so that
\[
g^{tt} e^w_{aux} +  e^w_{aux}g^{tt} =  Q_{-1}^w 
\]
This is possible since the coefficient $g^{tt}$ of $\tau^2$ in $p$ is
a scalar function which is nonzero near $r = 3M$. Also as defined
$e^w_{aux} \in OPS^{-3}$ and has kernel supported near $r = 3M$.

Having insured that the $D_t^3$ term does not appear, we
 divide $IQ^K[u,\tilde S,\tilde E]$ into two parts,
\[
IQ^K[u,\tilde S,\tilde E] = IQ^K_{princ}[u,\tilde S,\tilde E]
+  IQ^K_{aux}[u,\tilde S,\tilde E]
\]
where the main component is given by 
\begin{equation}
IQ^K_{princ}[u,\tilde S,\tilde E]  = \int_D Q_{2,p}^w
  u \cdot \overline u + 2\Re Q_{1,p}^w u \cdot \overline {D_t u} +
Q_{0,p}^w {D_t u} \overline {D_t u} \ dV_K
\label{qksform}\end{equation}
with operators $Q_{2,p}^w$, $Q_{1,p}^w$ and $Q_{0,p}^w$ defined by
\[
Q_{2,p}^w + 2 Q_{1,p}^w D_t + Q_{0,p}^w
 D_t^2 = \chi \left( \frac{1}{2i}
  \{p,s\} + pe \right)^w \chi
\]
while the remainder is given by a similar expression with operators
$Q_{2,a}^w$, $Q_{1,a}^w $ and $Q_{0,a}^w$ whose principal symbols are
supported away from $r = 3M$. More precisely, we have 
\[
Q_{2,a}^w + 2 Q_{1,a}^w D_t + Q_{0,a}^w D_t^2  - 
\left( \frac{1}{i} \chi s \{p,\chi\}\right)^w \in  OPS^0 + OPS^{-1}
D_t + OPS^{-2} D_t^2
\]
Hence, using the fact that the $LEW^1_K$ norm is nondegenerate outside
an $O(a)$ neighborhood of $3M$ we can bound in an elliptic fashion
\begin{equation}
|IQ^K_{aux}[u,\tilde S,\tilde E]| \lesssim \| u\|_{LEW^1_K}^2
+ \|D_t u\|_{H^{-1}_{comp}}^2
\label{iqaux}\end{equation}
where the last term on the right represents the $H^{-1}$ norm of $D_t$
$u$ in a compact region in $r$ (precisely, a neighborhood of $3M$).

In order to conclude the proof of the theorem we turn our attention to
the bound \eqref{mek}, which we seek to establish with $S$ and $E$
replaced with $\tilde S$, respectively $\tilde E$. We will show that
\begin{equation}
  \int_D  Q^K[u,X,q,m]  dV_K   + a  IQ^K_{princ}[u,\tilde S,\tilde E] 
\gtrsim \|u\|_{LEW_K^1}^2  -  O(a) \|D_t u\|_{H^{-1}_{comp}}^2
\label{meka}\end{equation}
We decompose the left hand side of \eqref{meka} into an outer part and
an inner part,
\[
LHS\eqref{meka} =LHS\eqref{meka}_{out} + LHS\eqref{meka}_{in}
\]
where
\[
\begin{split}
LHS\eqref{meka}_{out}= &\ \int_D \chi_o^2 Q^K[u,X,q,m] dV_k \\
LHS\eqref{meka}_{in}= &\  
 \int_D \chi^2 Q^K[u,X,q,m] dV_k + a I Q_{princ}^K[u,\tilde S,\tilde E] 
\end{split}
\]
For the first part we use the pointwise positivity of $Q^K$ away from 
$3M$ (see \eqref{qkbd}) to conclude that 
\begin{equation}
LHS\eqref{meka}_{out}  \gtrsim \int_D 
\chi_o^2 ( r^{-2} |\nabla u|^2 + r^{-4}|u|^2) d V_K
\label{mekout}\end{equation}
The second part is a quadratic form which for convenience we fully
recall here (see \eqref{qkform} and \eqref{qksform}:
\[
\begin{split}
LHS\eqref{meka}_{in} = &\ \int_D \chi^2 ( q^{K,\alpha \beta}
\partial_{\alpha} u \partial_\beta u + q^{K,0} u^2) dV^K
\\ &\ + a \int_D Q_{2,p}^w
  u \cdot \overline u + 2\Re Q_{1,p}^w u \cdot \overline {D_t u} +
Q_{0,p}^w {D_t u} \overline {D_t u} \ dV_K
\end{split}
\]
where the coefficients $q^{K,\alpha \beta}$, $q^{K,0}$ and 
operators $Q_{j,p}^w$ satisfy
\[
q^{K,\alpha \beta} \eta_\alpha \eta_\beta = \frac{1}{2i}\{ p,X\}+ q p, \qquad
q^{K,0} > 0
\]
respectively 
\[
Q_{2,p}^w + 2 Q_{1,p}^w D_t + Q_{0,p}^w D_t^2 = \chi \left( \frac{1}{2i}
  \{p,s\} + pe \right)^w \chi
\]
We carefully observe that in the two parts of the expression for
$LHS\eqref{meka}_{in}$ the cutoff function $\chi$ appears in different
places. In the first part, it is applied before the differentiation,
while in the second part it is applied before the pseudodifferential
operator. It does not make much sense to commute at this point. In the
first part, we would produce lower order terms which may significantly
alter $q^{K,0}$. In the second part, we would lose the compact
support of the kernels for the operators $Q_{j,p}^w$.

Since $s$ and $e$ are chosen as in Lemma~\ref{squaresum}, it follows
that the principal symbol for  $LHS\eqref{meka}_{in}$ admits the sum of
squares representation \eqref{sumsqK}. We want to translate this into
a sum of squares decomposition for  $LHS\eqref{meka}_{in}$.
However, some care is required due to the different positions of the
cutoff $\chi$, as explained above. The symbols $\mu_k= \mu_k(a)$ are in 
general of pseudodifferential type. However, part (i) of the Lemma
guarantees that in the Schwarzschild case they are of differential
type. Consequently, we write
\[
\mu_k(a) = \mu_k(0) + \mu_k(a) -\mu_k(0)
\]
and use this decomposition to define the pseudodifferential operators
\[
M_k = \chi \mu_k(0)(x,D) + (\mu_k(a) -\mu_k(0))^w \chi 
\]
Then using the Weyl calculus it follows that for $LHS\eqref{meka}_{in}$
we have the representation
\[
LHS\eqref{meka}_{in} = \int_D \sum_k |M_k u|^2 + q^{K,0} \chi^2 u^2 d
V_k + \int_{D} R_2^w u \cdot \bar{u} + 2 \Re R_{1}^w D_t u \cdot \bar
u + R_0^w D_t u \cdot \overline{D_t u} dV_K
\]
where the remainder terms satisfy $r_j \in a S^{j-2}$. What is
important here is that the remainder is of size $O(a)$. This follows
from our choice of the operators $M_k$, which guarantees that 
when $a = 0$ the remainder is zero.

Combining the last relation with \eqref{mekout} we obtain the bound
\[
\int_D \chi_o^2  r^{-2} |\nabla u|^2 + r^{-4} |u|^2  
+ \sum_k |M_k u|^2 dV_K \lesssim LHS\eqref{meka} + 
a(\| u\|_{L^2_{comp}}^2 + \|D_t u\|_{H^{-1}_{comp}}^2)
\]
where the last two terms on the right account for the remainder terms
involving the operators $R_j^w$, which can be bounded using norms 
of $u$ and $D_t u$ in a compact region in $r$, away from $r = 0$
and $r = \infty$. 

It is easy to see that the above left hand side dominates
$\|u\|_{LEW_K^1}$. For $r$ away one uses only the first two terms.  On
the other hand for $r$ close to $3M$ we use part (iii) of the Lemma,
which guarantees that the symbols $c_1(\tau-\tau_2)$,
$c_2(\tau-\tau_1)$ and $\xi$ can be recovered in an elliptic fashion
from the principal symbols $\mu_k$ of $M_k$. Thus \eqref{meka}
is proved. Together with \eqref{iqaux} this shows that
\[
\|u\|_{LEW_K^1}^2 \lesssim   
\int_D  Q^K[u,X,q,m]  dV_K   + a  IQ^K[u,\tilde S,\tilde E] 
+  O(a) \|D_t u\|_{H^{-1}_{comp}}^2
\]
The final step in the proof of \eqref{mek} is to establish that the
last  error term above is negligible. 
  We can  account for
it in an elliptic manner. Precisely, for any compactly supported
selfadjoint operator $Q \in OPS^{-1}$ we can use $Q^2$ in 
a Lagrangian term and integrate by parts
(commute) to obtain  
\[
\begin{split}
2 \Re \int_D (g^{tt})^{-1} \Box_K u\cdot  \overline{Q^2 u}\  dV_K =  
\|Q D_t u\|_{L^2}^2 + O( & \ \| Q D_t u\|_{L^2}\|u\|_{L^2_{comp}}^2
+  \|u\|_{L^2_{comp}}^2 \\ &\ + E[u](0)+ E[u](\tv_0))
\end{split}
\]
which leads to the elliptic bound 
\[
\|Q D_t u\|_{L^2}^2 \lesssim  \|u\|_{L^2_{comp}}^2 + \| \Box_K
u\|_{H^{-1}_{comp}}^2 + E[u](0)+ E[u](\tv_0)
\]
and further to
\[
\|D_t u\|_{H^{-1}_{comp}}^2 \lesssim  \|u\|_{L^2_{comp}}^2 + \| \Box_K
u\|_{H^{-1}_{comp}}^2 + E[u](0)+ E[u](\tv_0)
\]
Thus \eqref{mek} follows, and the  proof of the theorem is concluded.
\end{proof}

 Note that Theorem \ref{Kerr} tells us, in particular, that if we
 start with an initial data $(u_0, u_1)\in H^1\times L^2$ then
 $u(\tv)\in H^1$ is uniformly bounded for all $\tv>0$. 
A natural question to ask is if  this is also true for higher $H^n$
norms.  For $n \geq 1$ we define
\[
\|u\|_{LE^{n+1}_K} = \sum_{|\alpha| \leq n} \| \partial^\alpha
u\|_{LE^1_K}
\] 
respectively
\[
 \|f\|_{LE^{n*}_K} =  \sum_{|\alpha| \leq n} \|\partial^{\alpha}f\|_{LE^*_K}
 \]
The higher order energies are similarly defined,
\[
  E^{n+1}[u](\Sigma_R^{\pm}) =    
\sum_{|\alpha| \leq n} E[\partial^\alpha u](\Sigma_R^{\pm}), \qquad
  E^{n+1}[u]({\tv_0}) =    
\sum_{|\alpha| \leq n} E[\partial^\alpha u](\tv_0) 
\]

  We then have the following
 \begin{theorem}\label{highreg}
 Let $n$ be a positive integer and $u$ satisfy $\Box_K u=f$ with
 initial data $(u_0, u_1)\in H^{n+1} \times H^{n}$ on $\Sigma_R^-$
and $f\in LE_K^{n*}(\M_R)$. Then
 \[
 E^{n+1}[u](\Sigma_R^{+}) 
+ \sup_{\tv > 0}  E^{n+1}[u]({\tv_0}) 
+ \|u\|_{LE_K^{n+1}}^2 \lesssim  \|u_0\|_{H^{n+1}}^2 + \|u_1\|_{H^{n}}^2 +
 \|f\|_{LE^{n*}_K}^2
 \]
 \end{theorem}
 
\begin{proof}
We remark that by trace regularity results we have 
\[
\sum_{|\alpha| \leq n-1} \|\partial^\alpha f\|_{L^2(\Sigma_R^-)} 
\lesssim  \|f\|_{LE_K^{n*}}
\]
Since the initial surface $\Sigma_R^{-}$ is space-like, we can use the
equation to derive all higher $\tv$ derivatives of $u$ in terms of the
Cauchy data $(u_0,u_1)$ and $f$,
\[
 E^{n+1}[u](\Sigma_R^{-}) \lesssim \|u_0\|_{H^{n+1}}^2 + \|u_1\|_{H^{n}}^2 + \|f\|_{LE_K^{n*}}^2 
\]
Thus it suffices to prove that for $\tv_0 > 0$ we have
 \begin{equation}
 E^{n+1}[u](\Sigma_R^{+}) 
+  E^{n+1}[u]({\tv_0}) 
+ \|u\|_{LE^{n+1}_K}^2 \lesssim  E^{n+1}[u](\Sigma_R^{-}) +
 \|f\|_{LE_K^{n*}}^2
 \end{equation}
We will prove this for $n=2$, and the proof for the other cases
 will follow in a similar manner by induction.

 Since $\partial_\tv$ is a Killing vector field, we have $\Box_K
 (\partial_\tv u) = \partial_\tv f$. Then by Theorem \ref{Kerr} we 
 obtain
 \begin{equation}\label{t-est}
 E[\partial_\tv u](\Sigma_R^{+}) 
+  E[\partial_\tv u]({\tv_0}) 
+ \|\partial_\tv u\|_{LE^{1}_K}^2 \lesssim  E^{2}[u](\Sigma_R^{-}) +
 \| f\|_{LE_K^{1*}}^2
 \end{equation}

 In order to control the rest of the second order derivatives we
 take advantage of the equation, which takes the form
\begin{equation}\label{Kerreq}
 (g^{\tv \tv}\partial_{\tv\tv}+2g^{\tv\tphi}\partial_{\tv\tphi}+
L )u = f
\end{equation}
where $L$ is a spatial partial differential operator of order $2$.
This is most useful in the region where $\partial_\tv$ is time-like.
Given $\epsilon > 0$, this happens in the region of the form $r> 2M
+\epsilon$ provided that $a$ is sufficiently small. The fact that
$\partial_\tv$ is time-like is equivalent to the ellipticity of the
spatial part $L$  of $\Box_K$.   From \eqref{t-est} we obtain at $\tv = \tv_0$
\[
\| L u  \|_{L^2(\Sigma_{v_0})}^2
\lesssim E[u](v_0) + E[\partial_{\tv} u](v_0) + \| f\|_{L^2(\Sigma_{v_0})}^2
\]
The operator $L$ on the left is elliptic in $r \geq 2M+\e$, therefore
by a standard elliptic estimate we obtain
\[
E[\nabla u](\Sigma_{\tv_0} \cap \{r > 2M+\epsilon\})
\lesssim E[u](v_0) + E[\partial_{\tv} u](v_0) + \|
f\|_{L^2(\Sigma_{v_0})}^2
\]
A similar elliptic analysis leads to the corresponding local energy
bound,
\[
\begin{split}
  \| \nabla u\|_{LE^1_K(\M_R \cap \{r > 2M+2\epsilon\})}^2 &\ \lesssim
  \|u\|_{LE^1_K(\M_R)}^2 + \|\partial_{\tv} u\|_{LE^1_K(\M_R)}^2 + \|
  f\|_{L^2(\M_R)}^2 \\ &\ + E[\nabla u](\Sigma_{0} \cap \{r >
  2M+\epsilon\}) + E[\nabla u](\Sigma_{\tv_0} \cap \{r > 2M
  +\epsilon\})
\end{split}
\]
where the last two terms account for the output of integrations by
parts in  $\tv$.

We are left to deal with the case $r< 2M +2\e$, where $g^{rr}$ is
small and simply using the equation \eqref{Kerreq} does not suffice.
Let $\zeta(r)$ be a smooth cutoff function such that $\zeta = 1$ on
$[r_e, r_+ +2\e]$ and $\zeta = 0$ when $r> r_+ +3\e$. Then we
need bounds for the function $w = \zeta u$, which solves
\[
\Box_K w = \zeta f + [\Box_K,\zeta] u :=g
\]
The commutator above is supported in the region $\{ 2M + 2\epsilon
\leq r \leq 2M+ 3 \epsilon\}$ where we already have good estimates for
$u$. Recall that in the region $\{ r < 2M+3\epsilon\}$ the $LE^1_K$
and $LE^*_K$ norms are equivalent with the $H^1$, respectively $L^2$
norm. Hence it remains to prove that for all functions $w$ with
support in $\{ r < 2M+3\epsilon\}$ which solve $\Box_K w = g$ we have
\begin{equation}\label{t-estw}
  E[\nabla w](\Sigma_R^{+}) 
  +  E[\nabla w]({\tv_0}) 
  + \|\nabla w \|_{H^1(\M_R)}^2 \lesssim  E^{2}[u](\Sigma_R^{-}) +
  \| g\|_{H^1(\M_R)}^2
\end{equation}
This is an estimate which is localized near the event horizon,
and we will prove it taking advantage of the red shift effect.

Since $\partial_{\tv}$ is Killing, this bound  follows directly from
Theorem~\ref{Kerr} for the $\partial_{\tv} w$ component of $\nabla w$,
\begin{equation}\label{t-estwv}
  E[\partial_\tv w](\Sigma_R^{+}) 
  +  E[\partial_{\tv} w]({\tv_0}) 
  + \|\partial_{\tv} w \|_{H^1(D)}^2 \lesssim  E^{2}[w](\Sigma_R^{-}) +
  \| g\|_{H^1(D)}^2
\end{equation}
Consider now the angular derivatives of $w$, $\partial_{\omega} w$.
 We know that
\[
 [\Box_S, \partial_{\omega}]=0
\]
since the Schwarzschild metric is spherically symmetric. Hence by
\eqref{aproxmet} it follows that $[\Box_K,\partial_{\omega}]$ is a
second order operator whose coefficients have size $O(a)$.  Hence by
Theorem \ref{Kerr} we obtain
\begin{equation}\label{ang-est}
 E[\partial_{\omega} w](\Sigma_R^{+}) 
  +  E[\partial_{\omega} w]({\tv_0}) 
  + \| \partial_\omega w \|_{H^1(D)} \lesssim  E^{2}[w](\Sigma_R^{-}) +
  \| g\|_{H^1(D)}^2 + a \| \nabla_\omega w \|_{H^1 (D)}^2
\end{equation}

 We still need to bound $\partial_{r} w$. For that we 
 compute the commutator 
\[
 [\Box_K, \partial_r]w= -(\partial_r g^{rr})\partial_{rr}w + T w
\]
where $T$ stands for a second order operator with no $\partial_r^2$
terms.  The key observation, which is equivalent to the red shift
effect, is that $\partial_r g^{rr} > 0$ near $r = 2M$. Thus for 
$X$, $C$ and $q$ as in Lemma~\ref{ibp} we can
write the equation for $\partial_r w$ in the form 
\[
(\Box_K - \gamma [(X+CK)+q]) \partial_r w  = \partial_r g + Tw
\]
with $T$ as above and most importantly, a positive coefficient $\gamma$.
Because of this the operator 
\[
 B= \Box_K - \gamma[(X+CK)+q]
\]
satisfies the same estimate in Theorem \ref{Kerr} as $\Box_g$
for functions supported near the event horizon. 
Indeed, the same proof goes through as in Theorem \ref{Schwarz}.
Writing the integral identity  \eqref{intdivk} for $w$ we see 
the contribution of $\gamma$ is negative, therefore  we obtain the
inequality
\[
\int_D Q^K[w,X,q,m] dV_K \leq 
- \int_D  (\partial_r g + Tw)  \Bigl((X+CK)w + q w\Bigr) dV_K
+ BDR^K[w]
\]
By \eqref{qkbd} the left hand side is positive definite for $r <
2M+3\epsilon$. Using Cauchy-Schwarz for the first term on the right
and \eqref{bdrposk} for the second we obtain
\[
 E[\partial_{r} w](\Sigma_R^{+}) 
  +  E[\partial_{r} w]({\tv_0}) 
  + \| \partial_r w \|_{H^1(D)}^2 \lesssim  E^{2}[w](\Sigma_R^{-}) +
  \| \partial_r g + T w \|_{L^2 (D)}^2 
\]
Since $T$ contains no second order $r$ derivatives, this leads to
\begin{equation}\label{r-est}
 E[\partial_{r} w](\Sigma_R^{+}) 
  +  E[\partial_{r} w]({\tv_0}) 
  + \| \partial_r w \|_{H^1(D)}^2 \lesssim  E^{2}[w](\Sigma_R^{-}) +
  \| g\|_{H^1(D)} + a \| \nabla_{\omega,\tv} w \|_{H^1 (D)}^2
\end{equation}
Then the desired bound \eqref{t-estw} follows by combining
\eqref{t-estwv}, \eqref{ang-est} and \eqref{r-est} with appropriate
coefficients.

\end{proof}

  As an easy corollary, one obtains from Sobolev embeddings
 the pointwise boundedness result,
\begin{corr}  
 If $u$ satisfies $\Box_K u=0$ in $\M_R$ with
 initial data $(u_0, u_1)\in H^2\times H^1$ in $\Sigma_R^-$, then
\[
\|u\|_{L^\infty} \lesssim \|u_0\|_{H^2} + \|u_1\|_{H^1}
\] 
\end{corr}

\bigskip
%%%%%%%%%%%%%%%%%%%%%%%%%%%%%%%%%%%%%%%%%%%%%%%%%%%%%%%%%%%%%%%%%%%%%%%%%%%%%%%%%%%%%%%%%%%%%%%%
%%%%%%%%%%%%%%%%%%%%%%%%%%%%%%%%%%%%%%%%%%%%%%%%%%%%%%%%%%%%%%%%%%%%%%%%%%%%%%%%%%%%%%%%%%%%%%%%
%%%%%%%%%%%%%%%%%%%%%%%%%%%%%%%%%%%%%%%%%%%%%%%%%%%%%%%%%%%%%%%%%%%%%%%%%%%%%%%%%%%%%%%%%%%%%%%%

\end{document}